\documentclass[11pt]{amsart}
\usepackage{amsmath, amsrefs, amsthm, amssymb}
\usepackage{graphicx}
\usepackage{xcolor}
\usepackage{caption}
\usepackage{subcaption}
\usepackage{pdfsync}	
\usepackage{dsfont}
\usepackage{keyval} 
\usepackage{float}				%search for a pdf file
\newcommand\R{{\mathbb{R}}}

\newcommand\N{{\mathbb{N}}}

\newcommand\Z{{\mathbb{Z}}}

\newtheorem{theorem}{Theorem}[section]

\newtheorem{lemma}[theorem]{Lemma}

\theoremstyle{definition}

\theoremstyle{remark}

\numberwithin{equation}{section}

\begin{document}
\title[non-linear heat equation]
{Blowup for the nonlinear heat equation with small initial data in scale-invariant Besov norms}
\author{Lorenzo Brandolese}
\address{L. Brandolese: Universit\'e de Lyon, Universit\'e Claude Bernard Lyon 1, Institut Camille Jordan, UMR CNRS,  43 bd. du 11 Novembre 1918, 69622 Villeurbanne Cedex.}
\author{Fernando Cortez}
\address{F. Cortez: Escuela Polit\'ecnica Nacional del Ecuador, Facultad de Ciencias, Departemento de Matem\'atica,
Ladr\'on de Guevara E11-253, Quito-Ecuador.}
\date{\today}
\subjclass[2010]{35K55 (primary), 30H25 (secondary)}
\keywords{Blowup, Nonlinear Heat Equation, Besov}

\begin{abstract}
We consider the Cauchy problem of the nonlinear heat equation  $u_t -\Delta u= u^{b},\ u(0,x)=u_0$, with $b\geq 2$ and $b\in \mathbb{N}$. 
We prove that  initial data $u_0\in \mathcal{S}(\mathbb{R}^{n})$ (the Schwartz class)
arbitrarily small in the scale invariant Besov-norm
$\dot B^{-2/b}_{n(b-1) b/2,q}(\mathbb{R}^{n})$, can produce solutions that blow up in finite time. The case $b=3$ answers a question raised by Yves Meyer.
Our result also proves that the smallness assumption put in an earlier work by C.~Miao, B.~Yuan and B.~Zhang, for the global-in-time solvability, is essentially optimal.
\end{abstract}

\maketitle

\section{Introduction}
In this paper we study the Cauchy problem for the nonlinear heat equation
\begin{equation}
\label{Ter}
\begin{cases}
\partial_t u=
\Delta u  +  |u|^{\alpha} u, &    x\in \mathbb{R}^{n}, \quad t\in[0,T]\\ 
u(0,x) = u_0(x),
\end{cases}
\end{equation}
where $\alpha>0$, $0<T\leq\infty$, and $u\colon\R^{+}\times\R^{n} \to \R$ is a real function.
This problem attracted a considerable interest and we refer to, \emph{e.g.}, 
\cites{Ball, ColMR17, Ebd, Frie, Frie1, Grois, Henr, HerV94, MatM09, Suz, Weis} for a small sample of the huge existing literature.

Several well-posedness results are available for the Cauchy problem~\eqref{Ter}.
For example, if $u_0\in C_0(\R^n)$, then there is $T=T(u_0)>0$ and a unique
$u\in C([0,T),C_0(\R^n))$ which is a classical solution to~\eqref{Ter} on $(0,T)\times \R^n$.
For more singular data, say $u_0\in L^p(\R^n)$, we know the following, see  \cites{BreCaz,Weis1,Weis}.
\begin{enumerate}
\item[-]  When $p > \frac{n\alpha}{2}$ and $p\ge1$, there exists a constant $T=T(u_0)>0$ and a unique function $u(t)\in C([0,T],L^{p}(\mathbb{R}^{n}))$ that is a classical solution to~\eqref{Ter} on~$(0,T)\times \R^n$.
\item[-]  When $p < \frac{n\alpha}{2}$, there is no general theory of existence. Besides, A. Haraux and F. Weissler \cite{Haraux} established the non-uniqueness, by showing that there is a positive solution in $C([0,T],L^{p}(\mathbb{R}^{n})) \cap L^{\infty}_{loc}((0,T),L^{\infty}(\R^n))$, arising from zero initial data.
\item[-]  When $p = \frac{n\alpha}{2}$, see Theorem~\ref{localexis} below.
\end{enumerate}

We will be interested in the issues of the blowup in finite time \emph{v.s.} the global existence of the solutions.
The first works in this direction are due to H.~Fujita. 
Fujita proved that for the positive solutions of \eqref{Ter}, if the initial data $u_0$ is of class $C^{2}(\mathbb{R}^{n})$ with derivatives up to the second order bounded on $\mathbb{R}^{n}$, then a necessary condition for $u$ to be unique in $C(\mathbb{R}^{n}\times [0,T))$ is that
\begin{eqnarray*}
 \forall\,  x\in \mathbb{R}^{n}, \quad |u_0(x)| \leq M \ e^{|x|^\beta},
\end{eqnarray*}
for some constants $M>0$ and $0<\beta<2$.
See \cites{Fuji,Fuji1}.

About the problem of the existence of regular global solutions, there are two possible scenarios:
if  $n\alpha/2 < 1$, then no nontrivial positive solution of this problem can  be global (a situation now referred as \emph{Fujita's phenomenon}), while for $n\alpha/2 > 1$, there are global non-trivial positive solutions under small initial data assumptions.
K. Hayakawa \cite{Hay} and F. Weissler \cites{Weis1,Weis} later proved that Fujita's phenomenon occurs in the case of the critical exponent $n\alpha/2 =1$.

\section{Motivations and overview of the main result}

To motivate our results, we introduce the concept of a scale-invariant space. 
For $\lambda>0$, let us set 
\begin{eqnarray}
\label{scaling}
u_\lambda(t,x)= \lambda^{\frac{2}{\alpha}} u(\lambda^{2}t,\lambda x)  \qquad \mbox{and} \qquad u_{0,{\lambda}}(x)= \lambda^{\frac{2}{\alpha}} u_0(\lambda x).
\end{eqnarray}
For every solution $u(t,x)$ of \eqref{Ter}, $u_\lambda(t,x)$ is also a solution of  \eqref{Ter} with initial data $u_{0,{\lambda}}(x)$. 
In this setting, we say that a  Banach space $E$ is \emph{scale-invariant}, if
\begin{eqnarray}
\label{invsca}
\left \| u(t,\cdot) \right \|_{E}= \left \| u_\lambda(t,\cdot) \right \|_{E}. 
\end{eqnarray}  
Scale-invariant space are known to play an essential role in issues like well-posedness, global existence or blow-up of the solution. 

The purpose of the present paper is to study the borderline cases of explosion and global existence for solutions of~\eqref{Ter}, in a scale-invariant Banach space.
In the case of problem \eqref{Ter}, the only $L^p(\R^n)$-space invariant under the above scaling \eqref{scaling} is obtained for $p=n\alpha/2$. Notice that $p\geq 1$ if and only if $\alpha$ is larger or equal to the Fujita critical exponent. Therefore, we will be  especially interested
in solutions in  $L^{n\alpha/2}(\mathbb{R}^{n})$.

Our starting point is the following theorem, where we collect some of the results of Brezis, Cazenave and Weissler, in this scaling invariant setting.
\begin{theorem}[See~\cite{Weis1}. See also~\cite{BreCaz} for the uniqueness]
\label{localexis}
Let $u_0\in  L^{n\alpha/2}(\mathbb{R}^{n})$,
and assume that $\mbox{$n\alpha/2>1$}$.
There exists a time $T=T(u_0)>0$ and a function 
$u\in C([0,T],L^{n\alpha/2}(\mathbb{R}^{n})) \cap L^{\infty}_{loc}((0,T], L^{\infty}(\mathbb{R}^{n}))$
such that~$u$ is a classical solution of~\eqref{Ter} on $(0,T)\times\R^n$.
Moreover,
\begin{itemize}
\item [(i)] $\sup_{0<t<T}\, t^{\sigma/2}  \left \| u(\cdot,t) \right \|_{p}  < +\infty$,
\item [(ii)]  $\lim_{t\to0} \,t^{\sigma/2}  \left \| u(\cdot,t) \right \|_{p}=0$,
\end{itemize}
for any $\frac{n\alpha}{2}<p< \frac{n\alpha(\alpha+1)}{2}$ and $\sigma=\frac{2}{\alpha}-\frac{n}{p}$.

The uniqueness of classical solutions to~\eqref{Ter} holds in the class $C([0,T],L^{n\alpha/2}(\R^n))$.

Moreover, there exists $\delta=\delta(\alpha,n)$ such that if $\|u_0\|_{n\alpha/2}<\delta$ then 
such solution is global, \emph{i.e.}, one can take $T$ arbitrarily large.
\end{theorem}

The solution of Theorem~\ref{localexis} satisfies the integral equation
\begin{equation}
 \label{duhamel-alpha}
   u(t)= e^{t \Delta} u_0(x) +  \displaystyle\int_{0}^{t} e^{(t-s)\Delta}  \left | u \right |^{\alpha} u(s) \, ds, 
  \end{equation}
where $e^{t\Delta}f = G_t *f$ and
\[
G_t(x)= (4\pi t)^{-n/2}  \ e^{-|x|^{2}/(4t)}
\] 
is the standard Gaussian.
The uniqueness of~\emph{weak solutions} to the integral equation~\eqref{duhamel-alpha}
have been also addressed in~\cite{BreCaz}. The authors show that 
there is at most one \emph{weak solution} to~\eqref{duhamel} in the class
$C([0,T],L^{n\alpha/2}(\mathbb{R}^{n}))\cap L^{\infty}_{loc}((0,T], L^{\infty}(\mathbb{R}^{n}))$,
provided $\frac{n}{2}\alpha\ge \alpha+1>1$.
Under the additional restriction $\frac{n}{2}\alpha> \alpha+1>1$ the uniqueness of weak solutions
to~\eqref{duhamel} holds in the larger class $C([0,T],L^{n\alpha/2}(\mathbb{R}^{n}))$.
We refer to E.~Terraneo's paper \cite{Ter} for further uniqueness/non uniqueness results of weak solutions.

The problem of obtaining \emph{global-in-time} solutions by relaxing the stringent smallness assumption 
$\|u_0\|_{n\alpha/2}<\!\!\!<1$ was also addressed.
New ideas in this direction were brought by M.~Cannone and Y.~Meyer's works on the
Navier--Stokes equations~\cites{Canno1,Meyer}. 

In the model case $\alpha=2$ and $n=3$, \emph{i.e.} for the cubic heat equation in $\R^3$,
\[\partial_t u=\Delta u+u^3,\] 
Y.~Meyer observed in his lecture notes~\cite{Meyer} that if $u_0\in L^3(\R^3)$, with 
\[
\|u_0\|_{\dot B^{-1/2}_{6,\infty}}<\!\!\!<1,
\]
(this condition is considerably weaker than requiring $\|u_0\|_3<\!\!\!<1$)
then the maximal time $T^*$ of the solution is $T^*=+\infty$.
In fact, the method described therein would go through provided 
$\|u_0\|_{\dot B^{-1+3/p}_{p,\infty}}<\!\!\!<1$ and $3<p<9$.
In~\cite{Meyer}, he also raised the question whether or not, for $u_0\in L^3(\R^3)$, the even weaker smallness condition
\[
\|u_0\|_{\dot B^{-1}_{\infty,\infty}}<\!\!\!<1
\]
would still imply $T^*=+\infty$.
See next section for the definition of Besov spaces.
Notice that these Besov spaces enjoy the same scaling invariance properties as $L^3(\R^3)$ and we have the 
continuous injections
$
L^3(\R^3)\subset \dot B^{-1+3/p}_{p,\infty}(\R^3) \subset \dot B^{-1}_{\infty,\infty}(\R^3)
$
$(3<p\le+\infty)$.
Moreover, $\dot B^{-1}_{\infty,\infty}$ is known to be the largest function space invariant under translation and satisfying such scaling property. In this sense, a smallness condition on the $\dot B^{-1}_{\infty,\infty}$-norm would be the least demanding restriction that one could put in a scale-invariant setting. 

In the same spirit, but for the general case of problem~\eqref{Ter}, 
the best result for the global-in-time existence are due to Miao, Yuan, and Zhang~\cite{Mia}.
They proved (among other things) that the solution of Theorem~\ref{localexis} is global, provided $u_0\in L^{n\alpha/2}(\R^n)$, with $n\alpha/2>1$, 
under the smallness condition
\[
\|u_0\|_{\dot B^{-2/\alpha+n/p}_{p,q}}<\!\!\!<1,
\textstyle\qquad\text{for some $1<\frac{n\alpha}{2}<p<\frac{n\alpha(\alpha+1)}{2}$}, \quad 1\le q\le\infty.
\]
The restriction $\frac{n\alpha}{2}<p$, together with the condition $q\ge p$, ensure the embedding of 
$L^{n\alpha/2}(\R^n)$ into $\dot B^{-2/\alpha+n/p}_{p,q}(\R^n)$.
On the other hand the authors of~\cite{Mia} left open the limit case $p=n\alpha(\alpha+1)/2$.
In other words, they left open 
the question whether or not initial data $u_0\in L^{n\alpha/2}(\R^n)$, small in the $\dot B^{-2/(\alpha+1)}_{n\alpha(\alpha+1)/2,q}$-norm, give rise to global-in-time solutions.

Our main result below  provides a negative answer to the above problem, thus settling the borderline problem of the global solvability of~\eqref{class}, at least in the case of integer nonlinearity exponents.

More specifically, for $b\in\N$, we consider the Cauchy problem for the non-linear heat  equation  
\begin{equation}
\label{class}
\begin{cases}
\partial_t u  = \Delta u  + u^{b}\\
u(0,x) = u_0(x)
\end{cases} 
\qquad x\in \mathbb{R}^{n},  t\in(0,T). 
\end{equation} 
where  $0< T \leq \infty$.
The results recalled for the problem~\eqref{Ter} ---in particular Theorem~\ref{localexis}---
remain valid for~\eqref{class}, with $b=\alpha+1$. These two Cauchy problems in fact agree for positive solutions, or for real solutions of any sign, when $b$ is an odd integer.

\begin{theorem}
\label{inflation}
For any $\delta>0$ and $b< q \leq +\infty$, with $b\in \N$ and $n(b-1)/2>1$,
there exists
$u_0\in \mathcal{S}(\mathbb{R}^{n})$  (the Schwartz class)
such that
\begin{equation}
\label{u0small}
 \left \| u_{0} \right \|_{\dot{B}^{-2/b}_{nb(b-1)/2,q}} \leq \delta, 
\end{equation} 
and such that the maximal time~$T^*$ of the solution $u\in C([0,T^{*}),L^{n(b-1)/2}(\mathbb{R}^{n}))$ to~\eqref{class}  
arising from $u_0$ is finite.
\end{theorem}

As mentioned in the introduction, for $n(b-1)/2\le1$, because of Fujita's phenomenon, finite time blow up occurs for positive solutions, no matter which norms of the initial data are assumed to be small. 

In the case $b=n=3$, Theorem~\ref{inflation} 
negatively answers Y.~Meyer's question \cite[Conjecture 1]{Meyer}.

There are several blowup results for~\eqref{Ter} based on the maximum principle, energy functionals, concavity methods, or the spectral properties of the Laplacian, etc. See e.g.~\cite{BanB} for a review of these classical methods. 
But none of them seems to be effective to establish Theorem~\ref{inflation}, as the smallness condition~\eqref{u0small} represents a severe obstruction for their applicability.

The proof of Theorem~\ref{inflation} is constructive: suitable initial data are given by~\eqref{constru} below with $N=N(\delta)$ large enough. 
Our approach, inspired by Palais~\cite{Pal88}, rather uses the positivity of the Fourier transform inherited from its initial condition $u_0$. Even though conceptually similar to \cite{Pal88}, our paper is technically completely different 
(for example, we are able to remove the restriction $b\le 1+2/n$ that appears therein).
From the technical point of view, our paper is somehow closer to~\cites{Montg,Lema}, where the authors studied the blowup for different equations, namely 
diffusion problems with \emph{nonlocal quadratic} nonlinearity.
Our method bears also some relation with that of~\cite{RuC15}. However, the blowup result in 
$\mathcal{F}(L^1)$ of~\cite{RuC15} is not put in relation with the size of the data in scale-invariant norms. As such, our blowup result looks more precise, and its proof shorter.

Since the work of Cannone~\cite{Canno1} we know that fast enough oscillations of the initial give rise
to global-in-time smooth solutions for a large class of semilinear dissipative system, and that size conditions on Besov norms with negative regularity represent an effective way to measure such oscillations. The main interest of our result is to illustrate a limitation of this principle, by showing that there are scale invariant Besov norms that turn out to be too weak to be used for this purpose.

%In the more difficult case of the the Navier--Stokes equations, a conclusion weaker but somehow similar to  Theorem~\ref{inflation} 
%was obtained by Bourgain and Pavlovi\'c \cite{Bourg} (see also \cite{Yone}). These authors considered the Cauchy problem for Navier--Stokes with small data in~$\dot B^{-1}_{\infty,\infty}$. 
%While they left open the hard problem of the blowup, they succeded in obtaining a ``norm-%inflation'' phenomenon for the solution after an arbitrarily short time.

In our blowup result, the maximal time $T^*$ can be taken arbitrarily small, 
as one easily checks applying Theorem~1 to rescaled data $u_{0,\lambda}$,
that have the same Besov norm as in~\eqref{u0small}, and existence time 
$T^{*}_\lambda=\lambda^{-2}T^*$.

In the more difficult case of the Navier--Stokes equations, a similar problem
was addressed by Bourgain and Pavlovi\'c \cite{Bourg} (see also \cite{Yone}).
These authors considered the Cauchy problem for Navier--Stokes with small data in~$\dot B^{-1}_{\infty,\infty}$. While they left open the hard problem of the blowup, they succeeded in constructing a solution featuring a ``norm-inflation'' phenomenon in such Besov space,
after an arbitrarily short time.
But it was later realized by O. Sawada~\cite{Saw} that Bourgain's--Pavlovi\'c solution, in fact, \emph{does not blow up} in finite time.

\section{Preliminaries}

Let us recall the definition of the Besov norms and the Littlewood-Paley decomposition: 
let $\psi \in \mathcal{S}(\mathbb{R}^{n})$ such that $\mbox{supp\,}  \widehat{\psi} \subset \{ 3/4 \leq \left | \xi \right | \leq 8/3\}$ and
\begin{align*}
 1=\displaystyle\sum_{j=-\infty}^{\infty}  \widehat{\psi}_j (\xi) \ \ \ (\xi \in \mathbb{R}^{n},\;\xi\not=0),
\end{align*}
where $\psi_j(x)=2^{nj} \psi(2^{j}x)$, $j\in \mathbb{Z}$.
Here and throughout, $\widehat{f}$ denotes the Fourier transform of ~$f$.
The homogeneous Besov spaces $\dot{B}^{s.p}_{q}$ can be defined as follows, at least for $s<n/p$ and
$1\leq p, q \leq \infty$, which will be our case (in this paper we will only deal with the case $s<0$):
\begin{align*}
\dot{B}^{s}_{p,q} 
&= \{ f\in \mathcal{S'}(\R^n)\colon f=\sum_{j\in\Z} (2^{js}\psi_j*f) \text{ in the $\mathcal{S}'(\R^n)$-sense, and}\;
 \left \| f \right \|_{\dot{B}^{s,p}_{q}} <  \infty  \},
\end{align*}
where, for  $1\le q<+\infty$,
\[
\left \| f \right \|_{\dot{B}^{s}_{p,q}} = \left(\displaystyle\sum_{j=-\infty}^{\infty}  \left \| 2^{js} \psi_j * f  \right \|_{p}^{q}\right)^{1/q}
\]
and $\|f\|_{\dot B^s_{p,\infty}}=\sup_{q\in\Z}\|2^{js}\psi_j*f\|_p$.

As mentioned before,
one can obtain in Theorem~\ref{localexis} the global existence of the solution,
dropping the smallness assumption on $\|u_0\|_{n(b-1)/2}$, and putting instead a smallness assumption 
on the $\dot B^{-2/(b-1)+n/p}_{p,q}$-norm of the data, which is weaker than the
$L^{n(b-1)/2}$-norm. 
Let us sketch a proof of this fact, following the arguments of~\cites{Meyer,Mia},
putting in evidence the admissible range for $p$, which is $n(b-1)/2<p<nb(b-1)/2$.

One rewrites Equation \eqref{class} in the equivalent   Duhamel formulation
\begin{equation}
\label{duhamel}
u(t,x) = e^{t \Delta} u_0(x) +  \displaystyle\int_{0}^{t} e^{(t-\tau)\Delta}  u^{b}(\tau,x) \, d\tau
=:\Phi(u)(t,x).	
\end{equation}

If $u_0\in L^{n(b-1)/2}(\R^n)$, then the solution $u\in C([0,T],L^{n(b-1)/2}(\R^n))$ of Theorem~\ref{localexis} (recall that $\alpha=b-1$) is obtained through the contraction
mapping theorem,
as the limit $u=\lim u_l$ of approximate solutions (where $u_1=e^{t\Delta}u_0$ and 
$u_{l+1}=\Phi(u_l)$, for $l=1,2,\ldots$)  in the $X$-norm, where
\[
\begin{split}
\|u\|_X
&:=\sup_{0<t<T}\|u(t)\|_{n(b-1)/2} \,\,+\, \sup_{0<t<T} t^{1/(b-1)-n/(2p)}\|u(t)\|_p\\
&=:\|u\|_Y+\|u\|_Z.
\end{split}
\]
Indeed, first notice that $e^{t\Delta}u_0\in X$ by standard heat kernel estimates.
Next, the key estimates for the nonlinear term
are the following:
\begin{equation}
\begin{split}
\Bigl\|\int_0^t e^{(t-s)\Delta}u^b(s)\,ds\Bigr\|_p
&\le C\int_0^t (t-s)^{-\frac n2(\frac bp-\frac1p)} \|u^b(s)\|_{p/b}\,ds\\
&\le C\|u\|_Z^b\int_0^t (t-s)^{-\frac n2(\frac bp-\frac1p)} s^{-b/(b-1)+nb/(2p)}\,ds\\
&\le C\|u\|_Z^b\,t^{-1/(b-1)+n/(2p)},
\end{split}
\end{equation} 
and
\begin{equation}
\begin{split}
\Bigl\|\int_0^t e^{(t-s)\Delta}u^b(s)\,ds\Bigr\|_{n(b-1)/2}
&\le C\int_0^t (t-s)^{-\frac n2(\frac bp-\frac{2}{n(b-1)})} \|u^b(s)\|_{p/b}\,ds\\
&\le C\|u\|_Z^b.
\end{split}
\end{equation}
These estimates are valid when
$1< n(b-1)/2<p<nb(b-1)/2$ (one also needs here $1<b\le p$, but the restriction $b\le p$ can be dropped after the solution is constructed, by interpolation).
These estimate ensure that
\[
\|\Phi(u)\|_X\le C\|u\|_Z^b.
\]
The Lipschitz estimates
\[
\|\Phi(u)-\Phi(v)\|_X\le C(\|u\|_Z^{b-1}+\|v\|_Z^{b-1})\|u-v\|_{Z},
\]
is established in a similar way.
But $\|u_0\|_{\dot B^{-2/(b-1)+n/p}_{p,\infty}}\simeq \|u_1\|_Z$ owing to the heat kernel characterization 
of Besov spaces (see~\cite{Canno1}).
Hence, starting with $u_0$ small enough in the $\dot B^{-2/(b-1)+n/p}_{p,\infty}$-norm
allow to construct a solution with maximal lifetime $T^*=+\infty$.

Without any smallness assumption, a well known variant~\cites{Weis1, BreCaz} of the above argument still allows to construct a solution $u=\lim u_l$ in the $X$-norm, at least when $T>0$ is small enough. This relies on the observation that the approximate solutions (and hence the solution $u$ itself) 
satisfy the additional condition $\lim_{t\to0} t^{1/(b-1)-n/(2p)}\|u_l(t)\|_p=0$, for all~$l$.

\section{Proof of main theorem}
\label{sec:proof}

We start with a simple general remark about the properties of solutions of Theorem~\ref{localexis},
arising from initial data in the Schwartz class.
In this case, or more in general
when $u_0\in L^1\cap L^{n(b-1)/2}$, the corresponding solution obtained in Theorem~\ref{localexis}
remains in $L^1(\R^n)$ during the whole lifetime of the solution, and $u\in C([0,T],L^1(\R^n))$.
This could be seen applying Gronwall-type estimates, or otherwise with the following argument:
our claim is immediate if $b> n(b-1)/2$.
Indeed, in this case we may take $p=b$ in Theorem~\ref{localexis}, and we have
$\sup_{0<t<T} t^{1/(b-1)-n/(2b)}\|u(t)\|_b<\infty$. So,
\[
\|u(t)\|_1=\|\Phi(u)(t)\|_1\le \|u_0\|_1+C(u_0)\int_0^t s^{-b/(b-1)+n/2}\,ds\le C(u_0,T),
\]
because our condition $n(b-1)/2>1$ ensures that the above integral is finite for all finite $T>0$.
Moreover, the continuity with respect to~$t$ is obvious.
On the other hand, if $n(b-1)/2\ge b$ then
we first observe that
$u^b\in C([0,T],L^{n(b-1)/(2b)})$, next that $e^{t\Delta}u_0\in L^{n(b-1)/(2b)}$ by interpolation,
and from the integral equation $u(t)=\Phi(u)(t)$
we deduce
$u\in C([0,T],L^{n(b-1)/(2b)}\cap L^{n(b-1)/2})$.
If $n(b-1)/(2b^2)\le 1$, then using again the equation $u=\Phi(u)$ we get by interpolation 
that $u^b\in C([0,T],L^1(\R^n))$ and so $u\in C([0,T],L^1(\R^n))$. Otherwise we iterate this argument, until we find $m\in\N$ such that $n(b-1)/(2b^m)\le1$ and we conclude as before.

In the same way, going back to the sequence $(u_l)$ of approximate solutions introduced in
the previous section, one can prove that when
 $u_0\in L^1\cap L^{n(b-1)/2}$ not only the convergence $u_l\to u$ holds in the $X$-norm, but also $u_l\to u$ in the $C([0,T],L^1)$-norm, as $l\to\infty$. 

Later on we will choose a specific $u_0\in \mathcal{S}(\R^n)$ such that $\widehat u_0\ge0$ and $\widehat u_0$ is even (in a such way that $u$ is real-valued). All the approximate solutions
$u_l$ constructed from such datum $u_0$
satisfy $\widehat u_l(t,\cdot)\ge0$. The convergence of $(u_l)$ in the $C([0,T],L^1(\R^n))$-norm
implies that $\widehat u(t,\cdot)\ge0$ during the whole lifetime of the solution.

We introduce the following notation: for $b\in\mathbb{N}$, and a non-negative measurable function $f$, we denote
\begin{equation}
 f^{*b}=\underbrace{f*\ldots *f}_{b \text{ times}}.
\end{equation}

Let us now state a useful lemma.
\begin{lemma}
\label{lemme1}
Let  $\delta>0$, $b\in \mathbb{N}$ $(b\geq 2)$ and $w\in \mathcal{S}(\R^n)$,
such that $\widehat{w} \geq 0$.
Let $c_\delta=1-e^{-\frac{\delta}{2}(b^2 -1)}$.
Also assume that the support of
$\widehat{w}$ is contained in the ball $B(0,1)$. 
Let $w_k$,
$\alpha_k$ and $t_k$ be defined by the recursive relations $(k\ge1)$:
\[
\begin{cases}
w_0=w\\
w_k=w_{k-1}^b,
\end{cases}
\qquad
\begin{cases}
\alpha_0=1\\
\alpha_k=\alpha_{k-1}^b\,b^{-2k} \, c_\delta,
\end{cases}
\qquad
\begin{cases}
t_0=0\\
t_k = t_{k-1} +b^{-2k}\,\frac{\delta}{2}(b^{2}-1).
\end{cases}
\]
Then, if $u$ is the solution of \eqref{duhamel} with initial condition $u_0(x) \in L^{n(b-1)/2}(\mathbb{R}^{n})$, and if $\widehat{u}_0 (\xi) \geq A \widehat{w}(\xi)$ with $A > 0$, then, for any $k\in\N$, 
\begin{equation}
\label{rec:lem}
\widehat{u}(t,\xi) \geq A^{b^{k}} \alpha_k\,  e^{-b^{k}t} \mathds{1}_{t\geq t_k}  \ \widehat{w}_k(\xi), 
\end{equation}
where $\mathds{1}_{t\geq t_k}$ is the indicator function of the interval $[t_k,+\infty)$.
\end{lemma}

\begin{proof}
Using Fourier  transform, we have that \eqref{duhamel} becomes 
\begin{eqnarray}
\widehat{u}(\xi, t) = e^{-t  \left | \xi \right |^{2}} \widehat{u}_0(\xi) + \displaystyle\int_{0}^{t} e^{(s-t) \left | \xi \right |^{2}} 
[\widehat{u}(s,\xi)]^{*b}  \, ds.
\end{eqnarray}
We start with  the case $k=0$.
We have $\widehat u(t,\cdot)\ge0$, because $\widehat{u}_0(\xi) \geq A \widehat{w}(\xi) \ge 0$, as observed
at the beginning of this section. 
Then, using that $\mbox{supp\,}\widehat{w} \ \subset \{ \left | \xi \right | \leq 1\}$, we get
\begin{equation}
\begin{split}
\label{rec}
\widehat{u}(\xi, t) 
&\geq e^{-t  \left | \xi \right |^{2}} \widehat{u}_0(\xi) \geq e^{-t \left | \xi \right |^{2}} A \widehat{w}(\xi)\geq A   \ e^{-t }  \widehat{w}(\xi)\\
&=A\,\alpha_0\, e^{-t}\,\widehat w_0(\xi),
\qquad \forall  t \ge0.
\end{split}
\end{equation}
This agrees with~\eqref{rec:lem} for $k=0$.
Suppose now that inequality~\eqref{rec:lem} holds for $k-1$. Then we get, for all $t\geq t_k$:  
\begin{equation*}
\begin{split}
\widehat{u}(\xi, t)
&\geq \displaystyle\int_{0}^{t} e^{(s-t) \left | \xi \right |^{2}} [\widehat{u}(s,\xi)]^{*b} \, ds\\
&\geq \displaystyle\int_{t_{k-1}}^{t} e^{(s-t) \left | \xi \right |^{2}} (A^{b^{k-1}} \alpha_{k-1})^{b}
e^{-b^k s} \  [\widehat{w}_{k-1}(\xi)]^{*b}   \, ds.\\
&\geq 
  A^{b^{k}} \alpha_{k-1}^{b} \widehat{w}_k (\xi)  \displaystyle\int_{t_{k-1}}^{t} e^{(s-t) \left | \xi \right |^{2}} e^{-b^ks}  \,ds\\
&\geq 
 A^{b^{k}} \alpha_{k-1}^{b} \widehat{w}_k (\xi) e^{-b_k t} \displaystyle\int_{t_{k-1}}^{t} e^{(s-t)b^{2k}}\,ds,
\end{split}
\end{equation*}
where in the last inequality we used that $\mbox{supp\,}w_k\subset \{|\xi|\le b^{k}\}$.

But, for $t\ge t_k$, we have
\[
\begin{split}
\displaystyle\int_{t_{k-1}}^{t} e^{(s-t)b^{2k}}\,ds
 &=b^{-2k}(1-e^{-b^{2k}(t-t_{k-1})})\\
 &\ge b^{-2k}c_\delta,
\end{split}
\]
because $t_k - t_{k-1} =b^{-2k}\,\frac{\delta}{2}(b^{2}-1)$, 
and so $1- e^{-b^{2k}(t_k-t_{k-1})}= \ c_\delta$.

Hence we get,
\begin{equation*}
\begin{split}
\widehat{u}(\xi, t)
&\ge A^{b^{k}} \alpha_{k-1}^b \,b^{-2k} c_\delta\, e^{-b_kt} \mathds{1}_{t\ge t_k} \widehat{w_k}(\xi)\\
&\ge A^{b^{k}} \alpha_{k}\, e^{-b_kt} \mathds{1}_{t\ge t_k} \widehat{w_k}(\xi),
\end{split}
\end{equation*}
by the recursive relation defining $\alpha_k$.
Our claim now follows by induction.
\end{proof}

For later use, let us observe that closed form for the sequences introduced in the previous lemma
$w_k$, $\alpha_k$ and $t_k$, are
\[
w_k=w^{b^k} \qquad (k\ge0),
\]
next
\[
\alpha_k= b^{-\frac{2b}{(b-1)^2}b^k+\frac{2}{b-1}k+\frac{2b}{(b-1)^2}}\,
c_{\delta}^{\frac{b^{k}-1}{b-1}}\qquad (k\ge0),
\]
and
\[
t_k = \frac{\delta}{2}(b^2 -1 )  \,\sum_{j=1}^k b^{-2j} \qquad (k\ge1)
\]
as it is easily checked.

Next lemma provides a first blowup result for equation \eqref{duhamel}. 

\begin{lemma}
\label{lem:blo1}
Let  $\delta>0$ and $w\in\mathcal{S}(\mathbb{R}^{n})$ $(w\neq 0)$ be a  Schwartz  function such that $\widehat{w} \geq 0$. Also assume that the support of
$\widehat{w}$ is contained in the ball $B(0,1)$.
Let $u_0\geq Aw$, with 
$A\geq b^{2b/(b-1)^2} c_{\delta}^{-1/(b-1)}e^{\delta/2}\| \widehat{w} \|_{1}^{-1}$. 
If  $u$ is the classical solution of \eqref{duhamel} arising from $u_0$ and belonging to $C([0,T^{*}],L^{n(b-1)/2}(\mathbb{R}^{n}))$, then $T^{*}\leq \frac{\delta}{2}$.
\end{lemma}

\begin{proof}
Assume, by contradiction, $T^{*}>\frac{\delta}{2}$.
Applying Lemma \ref{lemme1}, and using that  $t_k \uparrow \frac{\delta}{2}$ as $k\rightarrow+\infty$, we get, for $t=\delta/2$ and all $k\in\N$,
\begin{align*}
\|\widehat u (\delta/2,\cdot)\|_{1}
&\geq  A^{b^{k}} \alpha_k e^{-b^{k}\delta/2} \| \widehat{w}_k \|_{1} \\
&= A^{b^{k}} \alpha_k e^{-b^{k}\delta/2} \| \widehat{w}\|_{1}^{b^k},
\end{align*}
by Tonelli's theorem and the non-negativity of $\widehat w$.
The size condition on $A$ ensures that, taking $\sup_{k\in\N}$ in the right-hand side,
one gets $\|\widehat u(\delta/2,\cdot)\|_1=+\infty$.
But by the positivity of $\widehat{u}(t,\cdot)$ and Fourier inversion formula,
\[
\|u(\delta/2,\cdot)\|_{L^\infty}
\ge (2\pi)^{-n}\|\widehat u (\delta/2,\cdot)\|_{L^1}=+\infty.
\]
This contradicts the fact that the lifetime of Weissler solution satisfies $T^*>\delta/2$.
\end{proof}

\begin{proof}[Proof of Theorem \ref{inflation}]
Let $\delta>0$ fixed and $w \in\mathcal{S}(\mathbb{R}^{n})$ such that $\widehat{w}\neq0$ and $\widehat{w}\geq 0$.
We also assume that $\widehat{w}$ is an even function and its support is contained in the ball
$B(0,{\frac{1}{2b}})$.
Let also $u_{0,N}\in \mathcal{S}(\mathbb{R}^{n})$ be defined as  
\begin{equation}
\label{constru}
u_{0,N} (x) = 
\epsilon_{N}  \displaystyle\sum_{k=0}^N   2^{2k/b} \eta_k 
\cos (\textstyle\frac32 2^{k} x_1)\, w(x)
\qquad (N\in\N),
\end{equation}
where the sequences $(\eta_k)$ and $(\epsilon_N)$ are chosen in the following way:
\begin{equation}
\label{choe}
\eta_k=1/(1+k)^{1/b},\qquad 
\epsilon_N=1/\log(\log(3+N)).
\end{equation}
In fact, the only thing that does matter in what follows
are the following properties of $(\eta_k)$ and $(\epsilon_N)$:
they must be nonnegative and such that
$(\eta_k)\in \ell^q$ if and only if $q>b$, and $\epsilon_N\to0$, with $\epsilon_{\!N}^b\sum_{k=0}^N\eta_k^b\to+\infty$ as $N\to+\infty$.
In the case $b\ge3$ is odd
we will additionally need 
$\epsilon_{\!N}^b\sum_{k=0}^{N-1} \eta_k^{b-1}\eta_{k+1}\to+\infty$
The choice~\eqref{choe} does satisfy these requirements.

Observe that the Fourier transform of $\cos(\frac32 2^{k}x_1) w(x)$ is 
$\frac{1}{2}[\widehat w(\xi+\frac32\,2^ke_1)+\widehat w(\xi-\frac32\,2^ke_1)]$, which is contained in the union of two balls centered at $\pm \frac32\,2^ke_1$ and radius $1/(2b)$ (and hence in a single dyadic annulus). 
Let us consider the homogeneous Littlewood--Paley decomposition, 
$\sum_{\in \Z} \widehat \psi(2^{-j}\xi)=1$, for $\xi\not=0$, obtained using a radial function $\widehat\psi\in C^\infty_0(\R^n)$ which is supported in $\{\frac34\le|\xi|\le \frac83\}$ and constant equal to 1 in $\{\frac54\le|\xi|\le\frac74\}$.
If we denote by $\Delta_j f=\psi_j*f$ the Littlewood--Paley dyadic blocks,
then we see that
\[
\Delta_j \cos (\textstyle\frac322^{k} x_1)\, w(x)=0,\qquad \text{if $k\not=j$},
\]
and 
\[
\Delta_j \cos (\textstyle\frac322^{j} x_1)\, w(x)=\cos (\textstyle\frac32 2^{j} x_1)\, w(x).
\]
Thus, if $q\geq 1$ we get 
\begin{align*}
\left \| u_{0,N} \right \|_{\dot B^{-2/b}_{n(b-1) b/2,q}} 
&=   \epsilon_{N} \left( \sum_{j\in\mathbb{Z}}^{} 2^{-\frac{2}{b}q j} \left \| \Delta_j \ u_{0,N} \right \|^{q}_{n(b-1)b/2}  \right)^{\frac{1}{q}} \\
&=
  \epsilon_{N} \left( \sum_{j=0}^{N} 
   \eta_j^q \| w(x) \cos (\textstyle\frac32 2^{j} x_1)\|^{q}_{n(b-1) b/2} \right)^{\frac{1}{q}} \\
&\le
  \epsilon_{N} \left( \sum_{j=0}^{N}  \eta_j^q \| w \|^{q}_{n(b-1) b/2} \right)^{\frac{1}{q}}\\
& = \epsilon_{N} \left(\sum_{j=0}^{N}  \eta_j^q\right)^{1/q} \left \| w \right \|_{n(b-1)b/2}.
\end{align*}
Thus, for any fixed $\delta>0$ and $q> b$, we can find  $N_0 \in \mathbb{N}$ such that 
\begin{align}
\label{small-be}
\left \| u_{0,N_0} \right \|_{\dot{B}^{-2/b}_{n(b-1)b/2,q}}<\delta.
 \end{align}

Now, let $T_N^*$ be the maximal time of the solution obtained in Theorem~\ref{localexis}, arising
from $u_{0,N}(x)$. We denote $u_N$ this solution. We are going to prove that $T^*_N<+\infty$, and more precisely that $T^*_N<\delta$.

If $N\ge N_0$ and $T_{N}^{*}<\delta$ then there is nothing to prove. 
We thus pick $N\ge N_0$ and assume $T^{*}_{N}\geq \delta$. 
By the remark at the beginning of Section~\ref{sec:proof}, we have  $\widehat{u}_{N}(t,\xi)  \geq 0$ for all $t\in [0, T_{N}^{*})$. 
Thus, if $ 0<t < T_{N}^{*}$, we get

 \begin{align*}
\widehat{u}_{N} (t,\xi) 
&= e^{-t \left | \xi \right |^{2} } \widehat{u}_{0,N} (\xi)+\displaystyle\int_{0}^{t}  
   e^{-(t-s)\left | \xi \right |^{2}}  [\widehat{u}_{N}(s,\xi)]^{*b}  \, ds\\
& \geq  e^{-t \left | \xi \right |^{2} } \widehat{u}_{0,N} (\xi)
\end{align*}
We have
\[
\widehat u_{0,N}(\xi)=\epsilon_N\sum_{k=0}^N 2^{2k/b}\eta_k \,
\textstyle\frac12[\widehat w(\xi+\frac32 2^ke_1) +\widehat w(\xi-\frac32 2^ke_1)].
\]
Hence, using that the support of $\widehat w$ is contained in $\{|\xi|\le \frac14\}$,
\begin{align*}
\widehat{u}_{N} (t,\xi) 
&\ge \epsilon_N \sum_{k=0}^N \biggl(
  \underbrace{e^{-t\,2^{2k+2}}2^{2k/b}\eta_k\textstyle\frac12 \widehat w(\xi+\frac32 2^ke_1) }_{
  =A_k(t,    \xi)}
  +
   \underbrace{e^{-t\,2^{2k+2}}2^{2k/b}\eta_k\textstyle\frac12 \widehat w(\xi-\frac32 2^ke_1) }_{
  =B_k(t,    \xi)}
  \biggr).
\end{align*}
This implies that
\begin{align}
\label{disti}
(\widehat u_N)^{*b}(t,\xi)
&\ge \epsilon_{\!N}^b\sum_{k_1=0}^N\cdots\sum_{k_b=0}^N \biggl( (A_{k_1}+B_{k_1})*\cdots*(A_{k_b}+B_{k_b})(t,\xi)\biggr).
\end{align}
It is now convenient to distinguish two cases.

\paragraph{\bf The case $b$ even}
We bound from below~\eqref{disti} retaining just a few terms of the above summation:
\begin{align*}
(\widehat u_N)^{*b}(t,\xi) 
&\ge\epsilon_{\!N}^b \sum_{k=0}^N (A_k*B_k)^{*b/2}(t,\xi).
\end{align*}
But,
\[
\textstyle\widehat w(\cdot+\frac32 2^ke_1)*\widehat w(\cdot-\frac32 2^ke_1)
=\widehat w*\widehat w,
\]
hence,
\begin{align*}
\label{disti}
(\widehat u_N)^{*b}(t,\xi)
&\ge \epsilon_{\!N}^b\sum_{k=0}^N e^{-b\,t\,2^{2k+2}}2^{2k-b}\,\eta_k^b(\widehat w)^{*b}(\xi).
\end{align*}
Using that $\text{supp}(\widehat w^{*b})\subset B(0,1)$, we deduce that
\begin{align*}
\widehat u_N(t,\xi)
&\ge \int_0^t e^{(s-t)|\xi|^2}\widehat u_N(s,\cdot)^{*b}\,ds\\
&\ge\epsilon_{\!N}^b \sum_{k=0}^N 2^{2k-b}\,\eta_k^b
 \Bigl(\int_0^t e^{(s-t)-b\,s\,2^{2k+2}}\,ds\Bigr)(\widehat w)^{*b}(\xi)\\
&\ge \epsilon_{\!N}^b\sum_{k=0}^N \frac{2^{-b}\eta_k^b}{4b}\, 
 \Bigl(1-e^{-t(b\,2^{2k+2}-1)}\Bigr) e^{-t}(\widehat w)^{*b}(\xi) \\
&\ge \epsilon_{\!N}^b\Bigl(\sum_{k=0}^N \eta_k^b\Bigr)\frac{2^{-b}}{4b}\, 
 \bigl(1-e^{-t(4b-1)}\bigr) e^{-t}(\widehat w)^{*b}(\xi)  .
\end{align*}

Our choice of $(\eta_k)$ and $(\epsilon_k)$ ensure that
$\epsilon_{\!N}^b(\sum_{k=0}^N \eta_k^b)\to+\infty$ as $N\to+\infty$.
Now let us take $t=\delta/2$ and $N\ge N_0$ large enough in a such way that 
\[
\epsilon_{\!N}^b\Bigl(\sum_{k=0}^N \eta_k^b\Bigr)\frac{2^{-b}}{4b}\, 
 \bigl(1-e^{-\delta(4b-1)/2}\bigr)e^{-\delta/2}
 \ge 
 \frac{b^{2b/(b-1)^2} e^{\delta/2}}{c_{\delta}^{1/(b-1)}\| \widehat{w} \|_1^b}.
\]
Hence~Lemma~\ref{lem:blo1} applies and implies that
the lifetime of the solution
of $u_t=\Delta u + u^b$ arising from the initial datum $u_N(\delta/2,\cdot)$ must blow up before the time $\delta/2$.
By the uniqueness result of Theorem~\ref{localexis}, this implies that $T^*_N<\delta$.

\paragraph{\textbf{The case $b$ odd}}
In this case we can write $b=2m+3$, with $m\in\N$.
Going back to~\eqref{disti}, we bound this expression from below in the following way:
\begin{align*}
(\widehat u_N)^{*b}(t,\xi)
&\ge \epsilon_{\!N}^b\sum_{k=0}^{N-1} \Bigl((A_k*B_k)^{*m}*A_{k+1}*B_k*B_k\Bigr)(t,\xi).
\end{align*}
By the invariance of convolution products under translation,
$\text{supp}{(A_{k}*B_k)}$ is contained in $\{|\xi|\le 1/b\}$ and $\text{supp}{(A_{k+1}*B_k*B_k)}$ is contained in $\{|\xi|\le 3/(2b)\}$. Hence,
\[
\text{supp} \Bigl((A_k*B_k)^{*m}*A_{k+1}*B_k*B_k\Bigr)(t,\cdot)\subset
\{|\xi|\le 1/2\}.
\]
Hence,
\begin{align*}
\label{disti}
 (\widehat u_N)^{*b}(t,\xi)
&\ge \epsilon_{\!N}^b\sum_{k=0}^{N-1} e^{-(m+3)\,t\,2^{2k+3}}2^{2k-b}\,\eta_k^{2m+2}\eta_{k+1}\,\textstyle \widehat w^{*b}(\xi).
\end{align*}
Arguing as before, we obtain,
\begin{align*}
\widehat u_N(t,\xi)
&\ge \int_0^t e^{(s-t)|\xi|^2}(\widehat u_N)^{*b}(s,\cdot)\,ds\\
&\ge \epsilon_{\!N}^b\sum_{k=0}^{N-1} 2^{2k-b}\,\eta_k^{2m+2}\eta_{k+1}
 \Bigl(\int_0^t e^{(s-t)-(m+3)\,s\,2^{2k+3}}\,ds\Bigr)\widehat w^{*b}(\xi)\\
&\ge \epsilon_{\!N}^b\Bigl(\sum_{k=0}^{N-1} \eta_k^{b-1}\eta_{k+1}\Bigr)\frac{2^{-b}}{8(m+3)}\, 
 \bigl(1-e^{-t(8(m+3)-1)}\bigr) e^{-t}\widehat w^{*b}(\xi)  .
\end{align*}
But $\epsilon_{\!N}^b\sum_{k=0}^{N-1} \eta_k^{b-1}\eta_{k+1}\to+\infty$ as $N\to+\infty$ and therefore we can conclude taking $t=\delta/2$ and applying Lemma~\ref{lem:blo1}, exactly as in the case
$b$ even.
\end{proof}

\section{Conclusions}
The global-in-time solvability of the Cauchy problem for the nonlinear heat equation~\eqref{class} in $\R^n$ is usually obtained putting a smallness assumption on a suitable scale invariant norm of the initial data. However, in the present paper we proved that the scale-invariant norm of the Besov space $\dot{B}^{-2/b}_{nb(b-1)/2,q}$ is not suitable for this purpose: in fact, arbitrarily small
initial data in this space (or in any larger scale invariant space) can give rise to solutions that blow up in finite time.
While our method provides a few quantitative estimates on the solution, it gives little information on the nature of the blowup. The issue of the type of blowup has been thoroughly investigated, e.g., in
\cites{ColMR17, HerV94,MatM09}. 

Our result is sharp, in the sense that, for all $s>-2/b$, a smallness condition on the 
$\dot{B}^{s}_{p,q}$-norm of $u_0$ (with $s-n/p=-2/b-2/(b(b-1))$, to respect the scale invariance), which is
slightly more stringent, does ensure that the solution is globally defined.
On the other hand, the precise role of the third index~$q$ (that does not affect the scaling of the Besov norm) on this blowup issue is less clear:
as the proof of Theorem~\ref{inflation} requires $q>b$, and breaks down when $q=b$, the following open problem naturally arises. (Here $b$ does not need to be an integer):
\emph{Let $n(b-1)/2>1$ and $u_0$ a smooth and well decaying initial data as $|x|\to+\infty$; does the smallness assumption $\|u_0\|_{\dot{B}^{-2/b}_{nb(b-1)/2,b}}<\!\!\!<1$
imply that the solution of the Cauchy problem for $u_t=\Delta u+ |u|^{b-1}u$ in $\R^n$
is global in time~?}

\section*{Acknowledgement}
The authors are grateful to Paul Acevedo and to the Referee for their useful comments and suggestions.

Fernando Cortez was supported by Escuela Politécnica Nacional, Proyecto PII-DM-08-2016.

\end{document}